\documentclass[12pt]{article}
\usepackage[a4paper]{anysize}\marginsize{2cm}{2cm}{2cm}{2cm}
\pdfpagewidth=\paperwidth \pdfpageheight=\paperheight
\usepackage{amsfonts,amssymb,amsthm,amsmath,eucal,tabu,url}
\usepackage{pgf}
 \usepackage{array}
 \usepackage{pstricks}
 \usepackage{pstricks-add}
 \usepackage{pgf,tikz}
 \usetikzlibrary{automata}
 \usetikzlibrary{arrows}
 \usepackage{indentfirst}
\theoremstyle{plain}
\newtheorem{thm}{Theorem}[section]

\newtheorem{lemma}[thm]{Lemma}
\newtheorem{proposition}[thm]{Proposition}

\theoremstyle{definition}
\newtheorem{definition}[thm]{Definition}
\newtheorem{remark}[thm]{Remark}
\newtheorem{example}[thm]{Example}

\newtheorem{question}[thm]{Question}

\newtheorem{thevarthm}[thm]{\varthmname}

\newenvironment{varthm*}[1]{\trivlist\item[]{\bf #1.}\it}{\endtrivlist}

\renewcommand\geq{\geqslant}

\renewcommand\leq{\leqslant}

\newcommand\be{\begin{eqnarray*}}
\newcommand\ee{\end{eqnarray*}}

\newcommand\newop[2]{\def#1{\mathop{\rm #2}\nolimits}}
\newop\log{log}
\newop\ord{ord}
\newop\Gal{Gal}
\newop\SL{SL}
\newop\Bl{Bl}
\newop\mult{mult}
\newop\mass{mass}
\newop\div{div}
\newop\codim{codim}
\newop\sing{sing}
\newop\vdim{vdim}
\newop\edim{edim}
\newop\Ass{Ass}
\newop\size{size}
\newop\reg{reg}
\newop\satdeg{satdeg}
\newop\supp{supp}
\newop\Neg{Neg}
\newop\Nef{Nef}
\newop\Nefh{Nef_H}
\newop\Eff{Eff}
\newop\Zar{Zar}
\newop\MB{MB}
\newop\MBxC{MB\mathit{(x,C)}}
\newop\NnB{NnB}
\newop\Bigg{Big}
\newop\Effbar{\overline{\Eff}}

\def\keywordname{{\bfseries Keywords}}%
\def\keywords#1{\par\addvspace\medskipamount{\rightskip=0pt plus1cm
\def\and{\ifhmode\unskip\nobreak\fi\ $\cdot$
}\noindent\keywordname\enspace\ignorespaces#1\par}}
\def\subclassname{{\bfseries Mathematics Subject Classification
(2000)}\enspace}
\def\subclass#1{\par\addvspace\medskipamount{\rightskip=0pt plus1cm
\def\and{\ifhmode\unskip\nobreak\fi\ $\cdot$
}\noindent\subclassname\ignorespaces#1\par}}

\begin{document}
\title{Seshadri constants and special configurations of points in the projective plane}
\author{Piotr Pokora}
\date{\today}
\maketitle
\thispagestyle{empty}
\begin{abstract}
In the present note, we focus on certain properties of special curves that might be used in the theory of multi-point Seshadri constants for ample line bundles on the complex projective plane. In particular, we provide three Ein-Lazarsfeld-Xu-type lemmas for plane curves and a lower bound on the multi-point Seshadri constant of $\mathcal{O}_{\mathbb{P}^{2}}(1)$ under the assumption that the chosen points are not very general. In the second part, we focus on certain arrangements of points in the plane which are given by line arrangements. We show that in some cases the multi-point Seshadri constants of $\mathcal{O}_{\mathbb{P}^{2}}(1)$ centered at singular loci of line arrangements are computed by lines from the arrangement having some extremal properties.\\
\subclass{14C20, 14N20, 52C30}
\end{abstract}

%*****************************************************************************
\section{Introduction}
In this note, we would like to present some new techniques that can be applied in the context of Seshadri constants which measure the local positivity of line bundles on algebraic varieties. Let us recall that for a projective variety $X$ of dimension $\dim X = n$ and $L$ a nef line bundle, the multi-point Seshadri constant of $L$ at $r\geq 1$ points $x_{1}, ..., x_{r} \in X$ is defined as 
$$\varepsilon(X,L;\, x_{1}, ..., x_{r}) = \inf_{\{x_{1},...,x_{r}\} \cap C \neq \emptyset} \frac{L\cdot C}{\sum_{i=1}^{r} {\rm mult}_{x_{i}}C},$$
where the infimum is taken over all irreducible and reduced curves $C$ on $X$. We know that there exists an upper-bound on the multi-point Seshadri constant, namely
$$\varepsilon(X,L;\, x_{1}, ..., x_{r})  \leq \sqrt[n] {\frac{L^{n}}{r} }.$$
Next, we define 
$$\varepsilon(X,L; \, r) = \max_{x_{1}, ..., x_{r} \in X} \varepsilon(X,L; x_{1}, ...,x_{r}),$$
and an interesting result due to K. Oguiso \cite{Oguiso} tells us that the value $\varepsilon(X,L;\, r)$ is attained at a set of very general points $x_{1}, ..., x_{r}$, i.e., outside a countable union of proper Zariski closed subsets in $X^{r}$. For further details about Seshadri constants (and more), we refer to the following fantastic lecture notes by B. Harbourne \cite{HarbInd}.

In the note, we restrict our attention to the case of the complex projective plane $\mathbb{P}^{2}$ and special point configurations. There are some interesting results providing bounds on multi-point Seshadri constants for $\mathcal{O}_{\mathbb{P}^{2}}(1)$ when points are in (very) general position, for instance in \cite{Bau,Tai,HR09, KSS, Oguiso,Roe1,Roe2,Sz1,Sz2,Sz3}, but not much is known about the precise values of Seshadri constants of $\mathcal{O}_{\mathbb{P}^{2}}(1)$ when the points are not in very general position. So our aim here is to present some new techniques and actual values of Seshadri constants when point configurations are \emph{special}. Since the most special position of points is when $x_{1}, ...,x_{r}$ are in a line $\ell \in \mathbb{P}^{2}$, for any configuration of points $x_{1}, ..., x_{r} \in \mathbb{P}^{2}$ one always has
$$ \frac{1}{r} \leq \varepsilon (\mathbb{P}^{2}, \mathcal{O}_{\mathbb{P}^{2}}(1);\, x_{1}, ..., x_{r}) \leq \frac{1}{\sqrt{r}}.$$

We will present three different Ein-Lazarsfeld-Xu-type lemmas (cf. \cite{EL,Xu}) in the case when points are not in very general position. The first one follows from Orevkov-Sakai-Zaidenberg inequality which involves the logarithmic version of the Bogomolov-Miyaoka-Yau inequality, the second lemma follows from B\'ezout Theorem, and the last inequality can be obtained using Huh's result on the Milnor numbers of singularities of reduced and irreducible hypersurfaces in $\mathbb{P}^{N}$. This allows us to give a non-trivial lower bound on multi-point Seshadri constants for $\mathcal{O}_{\mathbb{P}^{2}}(1)$ in special cases -- it seems to us that our bound is the first one in this setting. In the second part of the note, we compute Seshadri constants of $\mathcal{O}_{\mathbb{P}^{2}}(1)$ at point configurations given by singular loci of reduced curves. Our approach allows us to see that using rather special point configurations in the plane the values of the associated Seshadri constants of $\mathcal{O}_{\mathbb{P}^{2}}(1)$ are rather close to those with respect to configurations of very general points. The key idea behind this article is to use combinatorial methods in the theory of Seshadri constants and we try to understand which combinatorial properties of point configurations can be useful in the context approximations of Seshadri constants in general (please consult concluding remarks).

In the note, we work exclusively over the complex numbers.
\section{Ein-Lazarsfeld-Xu-type lemmas for plane curves}
In this section, we present three different Ein-Lazarsfeld-Xu-type lemmas for irreducible and reduced curves in the complex projective plane. Our general assumption is the following:
\begin{center}
$(\bullet)$ \emph{curves are passing through all given points with multiplicities greater than or equal to $2$}. 
\end{center}
The key advantage of our approach is that provided lemmas do not depend on the restriction that points must be in \emph{very general position}.

We assume in this section that if ${\rm Sing}(C) = \{p_{1}, ..., p_{s}\}$ are the singular points of a curve $C \subset \mathbb{P}^{2}$, then the multiplicities $m_{i}(C) = m_{p_{i}}(C)$ form a weakly-decreasing sequence $m_{1} \geq ... \geq m_{s} \geq 2$. We use also $m_{i}$ instead of $m_{i}(C)$ if this is clear from the context. We also start with $s\geq 3$ in order to avoid trivialities.
\begin{proposition}
Let $C \subset \mathbb{P}^{2}$ be an irreducible and reduced curve of degree $d \geq 4$ having singular points $p_{1}, ...,p_{s}$, then one has
$$d^{2} - \sum_{i=1}^{s}m_{i}^{2} > \frac{3}{2}\bigg( d - \sum_{i=1}^{s} m_{i}\bigg).$$
\end{proposition}
\begin{proof}
Let us recall that for such a curve the following Orevkov-Sakai-Zaidenberg inequality (\cite[Theorem 1]{OZ95}, \cite[Theorem (1.2)]{Sakai}) holds:
$$\sum_{i=1}^{s} \bigg(2 + \frac{1}{m_{i}}\bigg) \mu_{i} \leq 2d^{2}-3d,$$
where $\mu_{i}$'s are the corresponding Milnor numbers. Since for isolated singularities one always has $\mu_{i} \geq (m_{i}-1)^{2}$ (see for instance \cite[Theorem 1.8]{Lin}), we obtain
$$\sum_{i=1}^{s} \bigg(2 + \frac{1}{m_{i}}\bigg)(m_{i}-1)^{2} \leq \sum_{i=1}^{s} \bigg(2 + \frac{1}{m_{i}}\bigg)\mu_{i} \leq 2d^{2} - 3d.$$
Since
$$\sum_{i=1}^{s} \bigg(2 + \frac{1}{m_{i}}\bigg)(m_{i}-1)^{2} = \sum_{i=1}^{s} \bigg(2m_{i}^{2}-3m_{i}+\frac{1}{m_{i}}\bigg),$$
this leads to 
$$\sum_{i=1}^{s} \bigg(2m_{i}^{2}-3m_{i}\bigg) < \sum_{i=1}^{s} \bigg(2m_{i}^{2}-3m_{i}+\frac{1}{m_{i}}\bigg),$$
and we finally obtain
$$\sum_{i=1}^{s} \bigg(2m_{i}^{2}-3m_{i}\bigg) < 2d^{2} - 3d,$$
which completes the proof.
\end{proof}
This inequality, analytic in nature, can be viewed as a first step towards an Ein-Lazarsfeld Xu-type lemma. Before we present our lemma, let us recall the following definition.
\begin{definition}
The gonality ${\rm gon}(C)$ of a complex curve $C$ is the smallest degree of a non-constant map, defined over the complex numbers, from $C$ to the complex projective line $\mathbb{P}^{1}$.
\end{definition}
\begin{lemma}[\cite{OF}]
Let $C$ be a complex reduced and irreducible plane curve of degree $d \geq 1$, then ${\rm gon}(C) \leq d$.
\end{lemma}
Proposition 2.1 and Lemma 2.3 give our first Ein-Lazarsfeld-Xu-type inequality.
\begin{lemma}
Let $C$ be an irreducible and reduced plane curve in the projective plane of degree $\deg C \geq 3$ having singular points $p_{1}, ..., p_{s}$, then
$$C^{2} = d^{2} > \sum_{i=1}^{s} m_{i}\bigg(m_{i}-\frac{3}{2}\bigg) + \frac{3}{2} {\rm gon}(C).$$
\end{lemma}
Now we present a different approach to obtain a stronger version of the lemma above.
\begin{proposition}
Let $C$ be an irreducible and reduced plane curve in the projective plane having singular points $p_{1}, ..., p_{s}$, then
$$d^{2} - d\geq  \sum_{i=1}^{s} m_{i}(m_{i}-1)$$
\end{proposition}
\begin{proof}
Let $C$ be an irreducible and reduced curve in the plane as in the theorem, and let us denote by $f$ the defining equation of $C$. Now we introduce a new curve $D$ defined by the equation $g := \frac{\partial{f}}{\partial{x}}$. We know that $\deg D = d-1$ and at each point $p_{i}$ we have $m_{i}(D) \geq m_{i}(C)-1$. Now by B\'{e}zout theorem, since $C$ is irreducible and reduced, one has
$$d(d-1) = C.D \geq \sum_{i = 1}^{s} m_{i}(D)m_{i}(C)  \geq \sum_{i=1}^{s} m_{i}(C)(m_{i}(C)-1),$$
which completes the proof.

Alternatively, one can use Pl\"{u}cker-Teissier formulae \cite[Theorem 7.2.2]{Wall} to obtain the same result.
\end{proof}
Proposition 2.5 and Lemma 2.3 give our second Ein-Lazarsfeld-Xu-type inequality.
\begin{lemma}
Let $C$ be an irreducible and reduced plane curve in the projective plane of degree $\deg C \geq 3$ having singular points $p_{1}, ..., p_{s}$, then
\begin{equation}
\label{ineq}
C^{2} = d^{2} \geq \sum_{i=1}^{s} m_{i}(m_{i}-1) +  {\rm gon}(C).
\end{equation}

\end{lemma}
Finally, let us present a third variation which follows directly from a recent result due to J. Huh \cite{Huh}.
\begin{lemma}
Let $C$ be an irreducible and reduced plane curve in the projective plane of degree $\deg C \geq 3$ having singular points $p_{1}, ..., p_{s}$. Let $o \in {\rm Sing}(C)$, then
$$C^{2} = d^{2} \geq m_{0} + \sum_{i=1}^{s} (m_{i}-1)^{2} +  2({\rm gon}(C) - 1).$$
\end{lemma}
\begin{proof}
Since $C$ is not a pencil of lines passing through a point, we can apply Huh's inequality \cite[Theorem~1]{Huh}, namely
$$(d-1)^{2} \geq m_{o} - 1 + \sum_{p \in {\rm Sing}(C)} \mu_{p},$$
where $\mu_{i}$'s are the corresponding Milnor numbers and $o\in C$ is a singular point.
\end{proof}

Now we would like to make use of (\ref{ineq}) in order to provide a lower-bound for Seshadri constants in special cases. 

Let $\mathcal{P}$ be a finite set of $s \geq 1$ points in the projective plane and assume that the Seshadri constant of $\mathcal{O}_{\mathbb{P}^2}(1)$ centered at $\mathcal{P}$ is computed by an irreducible and reduced curve having at each point $p_{i} \in \mathcal{P}$ multiplicity greater than $1$. The inequality (\ref{ineq}) provides that 
$$\deg (C) > \sqrt{\sum_{i}^{s}m_{i}(m_{i}-1)},$$
which leads us to
\begin{equation*}
\varepsilon(\mathbb{P}^{2}, \mathcal{O}_{\mathbb{P}^{2}}(1);\mathcal{P}) > \frac{\sqrt{\sum_{i}^{s}m_{i}(m_{i}-1)}}{\sum_{i=1}^{s} m_{i}} = \sqrt{\frac{\sum_{i}^{s}m_{i}(m_{i}-1)}{(\sum_{i=1}^{s}m_{i})^{2}}}.
\end{equation*}
Since $$s \sum_{i=1}^{s} m_{i}^{2} \geq \bigg(\sum_{i=1}^{s} m_{i}\bigg)^{2}$$
we get
$$\frac{\sum_{i}^{s}m_{i}(m_{i}-1)}{(\sum_{i=1}^{s}m_{i})^{2}} \geq \frac{1}{s}\cdot \frac{\sum_{i=1}^{s}m_{i}^2 - \sum_{i=1}^{s}m_{i}}{\sum_{i=1}^{s}m_{i}^2} \geq \frac{1}{2s},$$
which gives us at the end the following lower bound
\begin{equation}\label{bound}
\varepsilon(\mathbb{P}^{2}, \mathcal{O}_{\mathbb{P}^{2}}(1);\mathcal{P}) > \sqrt{\frac{1}{2s}}.
\end{equation}

It is natural to ask whether there exists a configuration of points for which the Seshadri constant of $\mathcal{O}_{\mathbb{P}^{2}}(1)$ is computed by an irreducible and reduced curve having at each point of the configuration multiplicity greater than one. As an example which presents this phenomenon, let us focus on the so-called Severi curves.

\begin{definition}[\cite{Severi}]
A reduced and irreducible plane curve $D \subset \mathbb{P}^{2}$ is called \emph{a Severi curve} if the number of nodes of $D$ is equal to $\delta:= \frac{(d-1)(d-2)}{2}$ and these are the only singular points of $D$.
\end{definition}
\begin{proposition}
\label{prop1}
Let $D \subset \mathbb{P}^{2}$ be a Severi curve of degree $d \geq 6$. Then
$$\varepsilon(\mathbb{P}^{2}, \mathcal{O}_{\mathbb{P}^{2}}(1); {\rm Sing}(D)) = \frac{d}{(d-1)(d-2)}.$$
\end{proposition}
\begin{proof}
First of all, observe that the Seshadri ratio obtained by $D$ is equal to $\frac{d}{(d-1)(d-2)}$. In order to compete the proof, suppose that there exists an irreducible and reduced curve $C$, distinct from $D$, of degree $e$ in the plane having multiplicities $m_{1}(C), ..., m_{\delta}(C)$ at the singular points $p_{1}, ...,p_{\delta} \in {\rm Sing}(D)$ such that
$$\frac{e}{\sum_{i=1}^{\delta} m_{i}(C)} < \frac{d}{(d-1)(d-2)}.$$
This implies

$$\sum_{i=1}^{\delta} m_{i}(C) > \frac{e}{d}(d-1)(d-2).$$ 
Our aim is to make use of B\'{e}zout's Theorem. Observe that 
$$D.C = e.d  \geq \sum_{i=1}^{\delta} 2 m_{i}(C) > 2 \frac{e}{d}(d-1)(d-2).$$
Simple manipulations give 
$$d^{2} - 6d + 4 < 0,$$
a contradiction with the fact that $d \geq 6$.
\end{proof}
It is natural to ask whether, if at all, our bound (\ref{bound}) is sharp. We check our result in the case of Severi curves. Denote by $f(x) = \frac{x}{(x-1)(x-2)}$ and $g(x) = \sqrt{\frac{1}{(x-1)(x-2)}}$. Then
$$\begin{array}{c|cccccccc}
   x    & 6      & 7      & 8      & 9      & 10     & 20     & 50     & 100     \\ \hline
   f(x) & 0.3    & 0.2333 & 0.1904 & 0.1607 & 0.1388 & 0.0584 & 0.0212 & 0.0103  \\ \hline
   g(x) & 0.2236 & 0.1825 & 0.1543 & 0.1336 & 0.1178 & 0.0540 & 0.0206 & 0.0101
   \end{array}.$$
It means that for Severi curves our bound (\ref{bound}) is asymptotically sharp.

At the end of this section, let us emphasize that since our curves are (presumably) highly singular, the first gonality bound is very coarse. If $m_{1}$ is a maximal multiplicity of $C$, then (for instance by \cite{OF}):
$${\rm gon}(C) \leq d - m_{1}$$
This observation allows us to improve our Ein-Lazarsfeld-Xu-type lemmas, in the case of (\ref{ineq}) we obtain
$$C^{2} = d^{2} \geq \sum_{i=2}^{s} m_{i}(m_{i}-1) + m_{1}^{2} + {\rm gon}(C).$$
\section{Special configurations of points in the plane and Seshadri constants}
In this section, we focus on certain point configurations in the projective plane which are given by singular loci of certain line arrangements. We compute the multi-point Seshadri constants of $\mathcal{O}_{\mathbb{P}^{2}}(1)$ centered at the singular loci of those configurations. 

Our main aim here is to understand how accurate is the following question.
\begin{question}
Let $\mathcal{L} \subset \mathbb{P}^{2}$ be a line arrangement and denote by ${\rm Sing}(\mathcal{L})$ the singular locus of $\mathcal{L}$. For a given line $\ell \in \mathcal{L}$ we denote by $s(\ell)$ the number of singular points from ${\rm Sing}(\mathcal{L})$ contained in $\ell$ and by $s(\mathcal{L}) = \max_{\ell \in \mathcal{L}} s(\ell)$. Is it true that
$$\varepsilon(\mathbb{P}^{2}, \mathcal{O}_{\mathbb{P}^{2}}(1); \, {\rm Sing}(\mathcal{L})) = \frac{1}{s(\mathcal{L})} \, ?$$
In particular, is it true that in the above setting the multi-point Seshadri constant is computed by one of lines from the arrangement?
\end{question}
We start our considerations with a quite specific class of line arrangements satisfying \textit{Hirzebruch's property}.
\begin{definition}
Let $\mathcal{L} \subset \mathbb{P}^{2}$ be a line arrangement. We say that $\mathcal{L}$ satisfies Hirzebruch's property if the number of lines is equal to $3n$ for some $n \in \mathbb{Z}_{>0}$ and each line from $\mathcal{L}$ intersects others at exactly $n+1$ points.
\end{definition}
Conjecturally, all line arrangements satisfying Hirzebruch's property are those defined by complex reflection groups, see for instance \cite[Conjecture~5.1]{Panov}.
\begin{proposition}
\label{Panov}
Let $\mathcal{L} = \{\ell_{1}, ..., \ell_{3n}\} \subset \mathbb{P}^2$ be a line arrangement satisfying Hirzebruch's property and such that all singular points have multiplicity greater than two. Then
$$\varepsilon(\mathbb{P}^{2}, \mathcal{O}_{\mathbb{P}^{2}}(1); \, {\rm Sing}(\mathcal{L})) = \frac{1}{n+1}.$$
\end{proposition}
\begin{proof}
If we take a line $\ell_{i}$ from the arrangement $\mathcal{L}$, the Seshadri ratio is given by $\frac{1}{n+1}$. Suppose that there exists an irreducible and reduced curve $D$ of degree $d$, different from $\ell_{i}$ for every $i$, having multiplicities $m_{1}(D), ..., m_{s}(D)$ at singular points $p_{1}, ..., p_{s} \in {\rm Sing}(\mathcal{L})$ such that 
$$\frac{d}{\sum_{i=1}^{s} m_{i} (D)} < \frac{1}{n+1}.$$
This implies, in particular, that 
$$\sum_{i=1}^{s} m_{i} (D) > d(n+1).$$
Denote by $n_{i}(\mathcal{L})$ the multiplicity of a point $p_{i} \in {\rm Sing}(\mathcal{L})$, then 
$$3nd = D.(\ell_{1} + ... + \ell_{3n})\geq \sum_{i=1}^{s}m_{i}(D)n_{i}(\mathcal{L}) \geq 3 \sum_{i=1}^{s}m_{i}(D) > 3d(n+1),$$
a contradiction.
\end{proof}
\begin{example}
There exists an infinite series of line arrangements satisfying Hirzebruch's property and such that all singular points have multiplicity greater than two. Let us recall that $n$-th CEVA arrangement of lines with $n\geq 3$ (see for instance \cite[Page~206]{BHH87}) is given by the linear factors of the following polynomial
$$Q(x,y,z) = (x^{n}-y^{n})(y^{n}-z^{n})(z^{n}-x^{n}).$$
We know that $n$-th CEVA arrangement $\mathcal{C}_{n}$ consists of $3n$ lines and has exactly $3$ points of multiplicity $n$ and $n^{2}$ triple points, and on each line we have exactly $n+1$ singular points. This gives us
$$\varepsilon(\mathbb{P}^{2}, \mathcal{O}_{\mathbb{P}^{2}}(1); {\rm Sing}(\mathcal{C}_{n})) = \frac{1}{n+1}.$$
\end{example}
\begin{example}
Let us recall that Klein's arrangement of lines $\mathcal{K}$ \cite{Klein} consists of $21$ lines with exactly $21$ quadruple points and $28$ triple points. Klein's arrangement possesses Hirzebruch's property, and 
$$\varepsilon(\mathbb{P}^{2}, \mathcal{O}_{\mathbb{P}^{2}}(1); \, {\rm Sing}(\mathcal{K})) = \frac{1}{8}.$$

Another interesting example is given by Wiman's arrangement of lines $\mathcal{W}$ \cite{Wiman96} which consists of $45$ lines and exactly $120$ triple points, $45$ quadruple points, and $36$ quintuple points. Wiman's arrangement also possesses Hirzebruch's property, and 
$$\varepsilon(\mathbb{P}^{2}, \mathcal{O}_{\mathbb{P}^{2}}(1); \, {\rm Sing}(\mathcal{W})) = \frac{1}{16}.$$
\end{example}
In the before mentioned paper \cite{Panov}, D. Panov proved that there are exactly four real line arrangements (i.e., line arrangements defined over the real numbers) satisfying Hirzebruch's property, and these are reflection arrangements of certain Coxeter groups. Let us now present a numerical description of these line arrangements:
\begin{enumerate}
\item[a)] three generic lines intersecting at exactly three double points;
\item[b)] the well-known $\mathcal{A}_{1}(6)$ arrangement of $6$ lines consisting of $4$ triple points, and $3$ double points;
\item[c)] the arrangement $\mathcal{A}_{1}(9)$ of $9$ lines (four sides of a square in $\mathbb{R}^{2}$, four symmetry axes of the square, plus line at infinity), consisting of $3$ quadruple points, $4$ triple points, and $6$ double points;
\item[d)] the arrangement $\mathcal{A}_{1}(15)$ of $15$ lines (five sides of a regular pentagon in $\mathbb{R}^{2}$, five axes of symmetry, and five diagonals of the pentagon) consisting of $6$ quintuple points, $10$ triple points, and $15$ double points.
\end{enumerate}

Now we compute Seshadri constants for $\mathcal{O}_{\mathbb{P}^{2}}(1)$ for point configurations given by a), b), c), and d) -- notice that we cannot apply Proposition \ref{Panov} due to the fact that the arrangements above possess double points as intersections.

In cases a) and b), it is easy to see that the Seshadri constants are equal to $\frac{1}{2}$ and $\frac{1}{3}$, respectively, and this can be done by ad hoc arguments, so now we show how to deal with c).

Our aim is to pick in an appropriate manner exactly four lines from the configurations $\mathcal{A}_{1}(9)$ in such a way that these lines are passing through all singular points at least once. This can be achieved by some choice of lines, for instance we can take a subconfiguration $\mathcal{L} = \{\ell_{\infty}, \ell_{1}, \ell_{2}, \ell_{3}\}$ consisting of the line at infinity and three vertical lines (lines defined by two parallel sides of the square and the corresponding symmetry axis). After possible relabeling of the indices of points, we have that $m_{i}= 1$ for $i \in \{1, ..., 12\}$ and $m_{13} = 4$.
We claim that
$$\varepsilon (\mathbb{P}^{2}, \mathcal{O}_{\mathbb{P}^{2}}(1); \, {\rm Sing}(\mathcal{A}_{1}(9))) = \frac{1}{4}.$$
Suppose that there exists an irreducible and reduced curve $D$ of degree $e$ having multiplicities $n_{1}, ..., n_{13}$ such that
$$\frac{e}{\sum_{i=1}^{13}n_{i}} < \frac{1}{4},$$
This gives us, in particular, that
$$\sum_{i=1}^{13}n_{i} > 4e.$$
Now
$$4e = D.(\ell_{\infty} + \ell_{1} + \ell_{2} + \ell_{3}) \geq n_{1} + ... + n_{12} + 4n_{13} \geq \sum_{i=1}^{13}n_{i} > 4e,$$
a contradiction.

In the last case d), our strategy is quite similar, we are going to pick $6$ appropriate lines in order to obtain vanishing along all $31$ singular points of $\mathcal{A}_{1}(15)$. Let us present our choice which is depicted in Figure \ref{fig} by dashed lines $L_{1},L_{2},L_{3},L_{4},L_{5}$, and $L_{6}$ -- notice that $L_{1}$ is not a line from the arrangement.
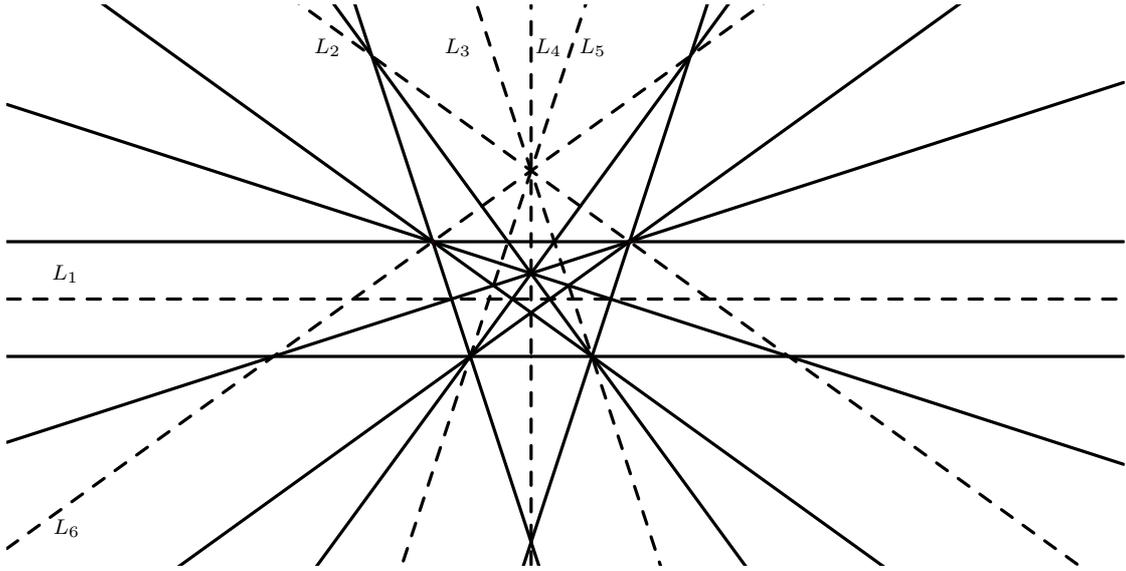
\begin{figure}[ht]
\centering
\definecolor{ttqqqq}{rgb}{0.2,0.,0.}
\begin{tikzpicture}[line cap=round,line join=round,>=triangle 45,x=1.0cm,y=1.0cm,scale=0.4]
\clip(-15.2361190567846,-6.910385475528476) rectangle (21.482657859816634,11.644778452114247);
\draw [line width=1.2pt,domain=-15.2361190567846:21.482657859816634] plot(\x,{(--19.91918627906224-1.9021130325903082*\x)/5.854101966249685});
\draw [line width=1.2pt,domain=-15.2361190567846:21.482657859816634] plot(\x,{(-0.--4.9797965697655595*\x)/3.618033988749895});
\draw [line width=1.2pt,dash pattern=on 5pt off 5pt] (2.,-6.910385475528476) -- (2.,11.644778452114247);
\draw [line width=1.2pt,domain=-15.2361190567846:21.482657859816634] plot(\x,{(--12.310734148701012--1.9021130325903055*\x)/5.854101966249685});
\draw [line width=1.2pt,domain=-15.2361190567846:21.482657859816634] plot(\x,{(--19.91918627906224-4.97979656976556*\x)/3.6180339887498945});
\draw [line width=1.2pt,domain=-15.2361190567846:21.482657859816634] plot(\x,{(-0.-3.8042260651806146*\x)/1.2360679774997894});
\draw [line width=1.2pt,domain=-15.2361190567846:21.482657859816634] plot(\x,{(-0.-0.*\x)/4.});
\draw [line width=1.2pt,domain=-15.2361190567846:21.482657859816634] plot(\x,{(-15.216904260722451--3.804226065180613*\x)/1.2360679774997898});
\draw [line width=1.2pt,dash pattern=on 5pt off 5pt,domain=-15.2361190567846:21.482657859816634] plot(\x,{(-24.621468297402025--2.3511410091698934*\x)/-3.2360679774997894});
\draw [line width=1.2pt,dash pattern=on 5pt off 5pt,domain=-15.2361190567846:21.482657859816634] plot(\x,{(--15.216904260722456--2.3511410091698917*\x)/3.23606797749979});
\draw [line width=1.2pt,dash pattern=on 5pt off 5pt,domain=-15.2361190567846:21.482657859816634] plot(\x,{(-0.-6.155367074350506*\x)/-2.});
\draw [line width=1.2pt,dash pattern=on 5pt off 5pt,domain=-15.2361190567846:21.482657859816634] plot(\x,{(--24.621468297402025-6.155367074350506*\x)/2.});
\draw [line width=1.2pt,domain=-15.2361190567846:21.482657859816634] plot(\x,{(--24.621468297402025-0.*\x)/6.47213595499958});
\draw [line width=1.2pt,domain=-15.2361190567846:21.482657859816634] plot(\x,{(-0.--3.804226065180613*\x)/5.23606797749979});
\draw [line width=1.2pt,domain=-15.2361190567846:21.482657859816634] plot(\x,{(--15.216904260722458-3.8042260651806146*\x)/5.23606797749979});
\draw [line width=1.2pt,dash pattern=on 5pt off 5pt,domain=-15.2361190567846:21.482657859816634] plot(\x,{(--9.95959313953112-0.*\x)/5.23606797749979});
\begin{scriptsize}
\draw [fill=black] (0.,0.) circle (1.0pt);
\draw [fill=black] (4.,0.) circle (1.0pt);
\draw [fill=ttqqqq] (5.23606797749979,3.804226065180613) circle (2.0pt);
\draw [fill=black] (2.,6.155367074350506) circle (2.0pt);
\draw [fill=ttqqqq] (-1.2360679774997894,3.8042260651806146) circle (2.0pt);
\draw [fill=black] (0.38196601125010554,4.97979656976556) circle (1.0pt);
\draw [fill=black] (3.618033988749895,4.9797965697655595) circle (1.0pt);
\draw [fill=black] (4.618033988749895,1.9021130325903064) circle (1.0pt);
\draw [fill=black] (2.,0.) circle (1.0pt);
\draw [fill=black] (-0.6180339887498947,1.9021130325903073) circle (1.0pt);
\draw[color=black] (2.568572238380157,10.252597336792869) node {$L_{4}$};
\draw[color=black] (-4.696872957203263,10.252597336792869) node {$L_{2}$};
\draw[color=black] (-13.267487948400472,-5.692226999622271) node {$L_{6}$};
\draw[color=black] (4.0260118434822205,10.252597336792869) node {$L_{5}$};
\draw[color=black] (-0.411565461604659,10.252597336792869) node {$L_{3}$};
\draw [fill=black] (2.,2.7527638409423467) circle (1.0pt);
\draw [fill=black] (1.23606797749979,3.804226065180614) circle (1.0pt);
\draw [fill=black] (2.7639320225002106,3.804226065180614) circle (1.0pt);
\draw [fill=black] (3.,3.0776835371752527) circle (1.0pt);
\draw [fill=black] (3.23606797749979,2.3511410091698917) circle (1.0pt);
\draw [fill=black] (2.6180339887498953,1.9021130325903068) circle (1.0pt);
\draw [fill=black] (2.,1.4530850560107211) circle (1.0pt);
\draw [fill=black] (1.3819660112501055,1.9021130325903073) circle (1.0pt);
\draw [fill=black] (0.7639320225002105,2.351141009169892) circle (1.0pt);
\draw [fill=black] (-6.472135954999584,0.) circle (1.0pt);
\draw [fill=black] (10.472135954999574,0.) circle (1.0pt);
\draw [fill=black] (2.,-6.155367074350508) circle (1.0pt);
\draw [fill=black] (2.,3.8042260651806132) circle (1.0pt);
\draw [fill=black] (-3.2360679774997885,9.959593139531119) circle (1.0pt);
\draw [fill=black] (7.236067977499792,9.959593139531123) circle (1.0pt);
\draw [fill=black] (1.,3.077683537175253) circle (1.0pt);
\draw[color=black] (-13.310993608254265,2.7478710120135745) node {$L_{1}$};
\end{scriptsize}
\end{tikzpicture}
\caption{Arrangement $\mathcal{A}_{1}(15)$ plus one additional line $L_{1}$. \label{fig}}
\end{figure}

Now we are aiming to show that 
$$\varepsilon (\mathbb{P}^{2}, \mathcal{O}_{\mathbb{P}^{2}}(1); \, {\rm Sing}(\mathcal{A}_{1}(15))) = \frac{1}{6}.$$
Suppose that there exists an irreducible and reduced curve $D$ of degree $e$ having multiplicities $n_{1}(D),..., n_{31}(D)$ at points ${\rm Sing}(\mathcal{A}_{1}(15)) = \{p_{1}, ...,p_{31}\}$ such that
$$\frac{e}{\sum_{i=1}^{31}n_{i}(D)} < \frac{1}{6}.$$ Moreover, we denote by $m_{i}(C)$'s the multiplicities of $C = L_{1} + ... + L_{6}$ at ${\rm Sing}(\mathcal{A}_{1}(15))$.
Observe that
$$D.C= 6e \geq \sum_{i=1}^{31}m_{i}(C)n_{i}(D) \geq \sum_{i=1}^{31}n_{i}(D) > 6e,$$
a contradiction.

The previous arrangements are quite rigid from a point of view of combinatorics and geometry, so at the end of the note we turn our attention towards very simple arrangements of lines.
\begin{definition}
We say that $\mathcal{A}_{d} \subset \mathbb{P}^{2}$ is a \textit{star arrangement} of $d$ lines if all intersection points are double points.
\end{definition}
It is easy to see that each line from $\mathcal{A}_{d}$ contains exactly $d-1$ double points from the arrangement.
\begin{proposition}
Let $\mathcal{A}_{d} \subset \mathbb{P}^{2}$ be a star arrangement of $d \geq 3$ lines. Then
$$\varepsilon(\mathbb{P}^{2}, \mathcal{O}_{\mathbb{P}^{2}}(1); {\rm Sing}(\mathcal{A}_{d})) = \frac{1}{d-1}.$$
\end{proposition}
\begin{proof}
This follows by (almost) the same argument as in Proposition \ref{Panov}.
\end{proof}
\begin{definition}
A line arrangement $\mathcal{H}_d \subset \mathbb{P}^2$ with $d\geq 3$ is called Hirzebruch's quasipencil if it consists of one point of multiplicity $d-1$ and $d-1$ double points. 
\end{definition}
We can assume that for Hirzebruch's quasipencil $\mathcal{H}_d= \{\ell_1, ... , \ell_d \}$ the point of multiplicity $d-1$ is defined by the intersection of $\ell_1 , ..., \ell_{d-1}$.

\begin{proposition}
Let $\mathcal{H}_{d} \subset \mathbb{P}^{2}$ be a Hirzebruch's quasipencil of $d \geq 3$ lines. Then
$$\varepsilon(\mathbb{P}^{2}, \mathcal{O}_{\mathbb{P}^{2}}(1); {\rm Sing}(\mathcal{H}_{d})) = \frac{1}{d-1}.$$
\end{proposition}
\begin{proof}
If we take $\ell_{d}$ from the arrangement $\mathcal{H}_{d}$, the Seshadri ratio is given by $\frac{1}{d-1}$. Let $D$ be an irreducible and reduced curve of degree $e$ such that $D\neq \ell_{i}$, for all $i \in \{1,...,d\}$, having multiplicities $m_1, ..., m_d$ at the points ${\rm Sing}(\mathcal{H}_d)$, and
$$\frac{e}{\sum_{i=1}^{d}m_{i}} < \frac{1}{d-1}.$$
Since
$$D(\ell_1 + \ell_d) = 2e \geq \sum_{i=1}^{d} m_i > e(d-1),$$
a contradiction.
\end{proof}
In our considerations the key role is played by information about the number of singular points on each line from the arrangement. It is not difficult to see that if $\mathcal{L} \subset \mathbb{P}^2$ is an arrangement of $d$ lines (we can assume that $\mathcal{L}$ is not a pencil of lines), then for each line $\ell_j$ the number of singular points $r_j := \#{\rm Sing}(\mathcal{L}) \cap \ell_j$ on $\ell_j$ is bounded from above by $d-1$, which can be easily verified using the following combinatorial equality
$$d-1 = \sum_{p \in {\rm Sing}(\mathcal{L}) \cap \ell_j} (m_p - 1).$$
However, we are not aware of "reasonable" lower bounds for $r_j$. This leads to the following simple question.
\begin{question}
Is it true that for an arrangement $\mathcal{L}\subset \mathbb{P}^2$ of $d\geq 3$ lines, which is not a pencil of lines, there exists a positive integer $C$ such that
$$\max_{j=1}^{d} r_j \geq \frac{1}{C} \cdot d ?$$
\end{question}
\section{Concluding remarks}
We now make a few remarks exhibiting some interesting similarities and differences between the behavior of Seshadri constants at very general points and special points in $\mathbb{P}^{2}$.
\begin{remark}
It might happen that for \emph{special configurations} of $s\geq 1$ points the Seshadri constant of $\mathcal{O}_{\mathbb{P}^{2}}(1)$ has the same value as in the case of $s$ very general points. Let us present an interesting example. In his Diplomarbeit, K. Ivinskis explained that there exists an irreducible and reduced curve $C_{d}$ of degree $d = 6k$ with $k\geq 1$ having exactly $9k^{2}$ ordinary cusps \cite[Lemma~4.1.7]{Ivinskis}. This curve is constructed by using the Kummer cover $\phi: \mathbb{P}^{2} \ni (x,y,z) \mapsto (x^{n},y^{n},z^{n}) \in \mathbb{P}^{2}$ which is branched along $xyz = 0$, please consult \cite[Section~7]{Persson} for details. It is easy to check that
$$\varepsilon (\mathbb{P}^{2}, \mathcal{O}_{\mathbb{P}^{2}}(1); {\rm Sing}(C_{d})) = \frac{1}{3k}.$$

Notice that we obtained exactly the same value as the Seshadri constant of $\mathcal{O}_{\mathbb{P}^{2}}(1)$ centered at configurations of $9k^{2}$ points with $k\geq 1$ in very general position. %The value of the Seshadri constant suggests that the set of cusps might be linearly independent, which is not known a priori.
\end{remark}
\begin{remark}
It was shown in \cite{Seredica}, irreducible and reduced curves $C \subset \mathbb{P}^{2}$ computing the multi-point Seshadri constants for very general points $\mathcal{P} = \{p_{1}, ..., p_{r}\}$ are homogeneous or almost-homogeneous. Let us recall that the sequence of multiplicities $(m_{1}, ..., m_{r})$ is almost-homogeneous if all but at most one of the coordinates are equal, and we say that a curve $C$ is almost-homogeneous at $\mathcal{P}$ if the $r$-tuple $(m_{1}(C), ..., m_{r}(C))$ is almost-homogeneous. Our aim is to present an example showing that for special point configurations irreducible and reduced curves computing Seshadri constants are in general \emph{not} almost-homogeneous.

It is well-known that there exists an irreducible and reduced plane sextic $C_{6} \in \mathbb{P}^{2}$ having exactly one triple point $p_{1}$ and $7$ double points $p_{2}, ...,p_{8}$. We define $\mathcal{P} = \{p_{1}, p_{2}, ..., p_{8},p_{9}, ..., p_{27}\}$, where $p_{9}, ..., p_{27} \in C_{6}$ are arbitrary smooth and mutually distinct points. 
It is easy to compute that
$$\varepsilon(\mathbb{P}^{2}, \mathcal{O}_{\mathbb{P}^{2}}(1); \mathcal{P}) = \frac{1}{6},$$
and the Seshadri constant is computed by $C_{6}$.
\end{remark}
\begin{remark}
As we observed in Introduction, for any configuration of points $x_{1}, ..., x_{r} \in \mathbb{P}^{2}$ one has
$$ \frac{1}{r} \leq \varepsilon (\mathbb{P}^{2}, \mathcal{O}_{\mathbb{P}^{2}}(1); x_{1}, ..., x_{r}) \leq \frac{1}{\sqrt{r}}.$$
Our experiment shows an interesting phenomenon that for line arrangements with Hirzebruch's property the values of Seshadri constants of $\mathcal{O}_{\mathbb{P}^{2}}(1)$ centered at singular loci of the arrangements are close to $\frac{1}{\sqrt{r}}$. Let us point out here that for Klein's arrangement of lines the Seshadri constant of $\mathcal{O}_{\mathbb{P}^{2}}(1)$ is equal to $\frac{1}{8}$, but for very general $49$ points we know that the Seshadri constant of $\mathcal{O}_{\mathbb{P}^{2}}(1)$ is equal $\frac{1}{7}$. In the case of $n$-th CEVA's arrangement, we have exactly $r = n^{2}+3$ singular points and the Seshadri ratio is equal to $\frac{1}{n+1}$, which is really close to the predicted value $\frac{1}{\sqrt{n^{2}+3}}$ for very general points. 
\end{remark}
\section*{Acknowledgement} 
The author would like to thank Joaquim Ro\'{e}, Mike Roth, and Tomasz Szemberg for discussions around Seshadri constants. I also would like to thank Szymon Brzostowski for explaining to me certain aspects of \cite{Huh}, and  Halszka Tutaj-Gasi\'nska for useful discussions, and for suggesting a choice of lines in the case d) on page $8$ which helped to compute the Seshadri constant. Finally, I would like to warmly thank an anonymous referee for valuable comments and remarks that allowed to improve this note.

%*****************************************************************************

    Institute of Mathematics,
    Polish Academy of Sciences,
    ul. \'{S}niadeckich 8,
    PL-00-656 Warszawa, Poland. \\

\nopagebreak
\textit{E-mail address:} \texttt{piotrpkr@gmail.com, ppokora@impan.pl}

\end{document}